\documentclass[psamsfonts]{amsart}

\usepackage{amssymb,amsfonts,verbatim}
\usepackage[all,arc]{xy}
\usepackage{enumerate}
\usepackage[all]{xy}
\usepackage{mathrsfs}
\usepackage{array}

\newtheorem{thm}{Theorem}[section]

\newtheorem{cor}[thm]{Corollary}

\newtheorem{prop}[thm]{Proposition}
\newtheorem{lem}[thm]{Lemma}

\theoremstyle{definition}
\newtheorem{defn}[thm]{Definition}
\newtheorem{defns}[thm]{Definitions}

\theoremstyle{remark}
\newtheorem{rem}[thm]{Remark}

\makeatletter
\let\c@equation\c@thm
\makeatother
\numberwithin{equation}{section}

\title{Refined Limit Multiplicity for Varying Conductor}

\author{John Binder}

\begin{document}
\newcommand{\NN}{\mathscr{N}}
\newcommand{\CC}{\mathbb{C}}
\newcommand{\DDD}{\mathscr{D}}
\newcommand{\HH}{\mathcal{H}}
\newcommand{\RR}{\mathbb{R}}
\newcommand{\RRR}{\mathscr{R}}
\newcommand{\FF}{\mathcal{F}}
\newcommand{\KK}{\mathscr{K}}
\newcommand{\UU}{\mathscr{U}}
\newcommand{\II}{\mathscr{I}}
\newcommand{\EE}{\mathscr{E}}
\newcommand{\GG}{\mathscr{G}}
\newcommand{\PZ}{\mathbb{P}_{\mathbb{Z}}}
\newcommand{\ZZ}{\mathbb{Z}}
\newcommand{\PP}{\mathscr{P}}
\newcommand{\PPP}{\mathscr{P}}
\newcommand{\SSS}{\mathcal{S}}
\newcommand{\LL}{\mathscr{L}}
\newcommand{\MM}{\mathscr{M}}
\newcommand{\AAA}{\mathbb{A}}
\newcommand{\GGG}{\mathscr{G}}
\newcommand{\AAAA}{\AAA}
\newcommand{\Df}{\mathcal{D}_F}
\newcommand{\QQ}{\mathbb{Q}}
\newcommand{\QQQ}{\mathscr{Q}}
\newcommand{\DD}{\mathscr{D}}
\newcommand{\OO}{\mathcal{O}}
\newcommand{\VV}{\mathscr{V}}
\newcommand{\pp}{\mathfrak{p}}
\newcommand{\qq}{\mathfrak{q}}
\newcommand{\faa}{\mathfrak{a}}
\newcommand{\mm}{\mathfrak{m}}
\newcommand{\IIII}{\mathcal{I}}
\newcommand{\JJ}{\mathscr{J}}
\newcommand{\weak}{\rightharpoonup}
\newcommand{\weaks}{\rightharpoonup^*}
\newcommand{\simga}{\sigma}
\newcommand{\linf}[2]{\left\langle {#1},\, {#2}\right\rangle}
\newcommand{\into}{\hookrightarrow}
\newcommand{\im}{\text{im}}
\newcommand{\lists}[3]{{#1}_1{#2}\ldots {#2}{#1}_{#3}}
\newcommand{\qr}[2]{\left(\frac{#1}{#2}\right)}
\newcolumntype{M}{>{$}c<{$}}
\newcommand{\alg}{\text{alg}}
\newcommand{\Spec}{\mathop{\mathrm{Spec}}\nolimits}
\newcommand{\Proj}{\mathop{\mathrm{Proj}}\nolimits}
\newcommand{\Tr}{\mathop{\mathrm{Tr}}\nolimits}
\newcommand{\Hom}{\mathop{\mathrm{Hom}}\nolimits}
\newcommand{\spa}{\mathop{\mathrm{sp}}\nolimits}
\newcommand{\rank}{\mathop{\mathrm{rank}}\nolimits}
\newcommand{\Pic}{\mathop{\mathrm{Pic}}\nolimits}
\newcommand{\image}{\mathop{\mathrm{Im}}\nolimits}
\newcommand{\tors}{\mathop{\mathrm{tors}}\nolimits}
\newcommand{\ord}{\mathop{\mathrm{ord}}\nolimits}
\newcommand{\Imag}{\mathop{\mathrm{Im}}\nolimits}
\newcommand{\trdeg}{\mathop{\mathrm{trdeg}}\nolimits}
\newcommand{\codim}{\mathop{\mathrm{codim}}\nolimits}
\newcommand{\hei}{\mathop{\mathrm{ht}}\nolimits}
\newcommand{\sgn}{\mathop{\mathrm{sgn}}\nolimits}
\newcommand{\Gal}{\mathop{\mathrm{Gal}}\nolimits}
\newcommand{\supp}{\mathop{\mathrm{supp}}\nolimits}
\newcommand{\Cl}{\mathop{\mathrm{Cl}}\nolimits}
\newcommand{\CaCl}{\mathop{\mathrm{Ca\,Cl}}\nolimits}
\newcommand{\Div}{\mathop{\mathrm{Div}}\nolimits}
\newcommand{\Sym}{\mathop{\mathrm{Sym}}\nolimits}
\newcommand{\coker}{\mathop{\mathrm{coker}}\nolimits}
\newcommand{\imag}{\mathop{\mathrm{im}}\nolimits}
\newcommand{\End}{\mathop{\mathrm{End}}\nolimits}
\newcommand{\Frob}{\mathop{\mathrm{Frob}}\nolimits}
\newcommand{\Ann}{\mathop{\mathrm{Ann}}\nolimits}
\newcommand{\Art}{\mathop{\mathrm{Art}}\nolimits}
\newcommand{\rec}{\mathop{\mathrm{rec}}\nolimits}
\newcommand{\Aut}{\mathop{\mathrm{Aut}}\nolimits}
\newcommand{\Ad}{\mathop{\mathrm{Ad}}\nolimits}
\newcommand{\nr}{\mathop{\mathrm{nr}}\nolimits}
\newcommand{\cond}{\mathop{\mathrm{cond}}\nolimits}

\newcommand{\Ind}{\mathop{\mathrm{Ind}}\nolimits}
\newcommand{\cInd}{\mathop{\mathrm{c-Ind}}\nolimits}

\newcommand{\spann}{\mathop{\mathrm{span}}\nolimits}
\newcommand{\Vol}{\mathop{\mathrm{Vol}}\nolimits}
\newcommand{\Irr}{\mathop{\mathrm{Irr}}\nolimits}
\newcommand{\Res}{\mathop{\mathrm{Res}}\nolimits}
\newcommand{\genus}{\mathop{\mathrm{genus}}\nolimits}
\newcommand{\scusp}{\mathop{\mathrm{scusp}}\nolimits}
\newcommand{\PProj}{\mathbb{P}\mathop{\mathrm{roj}}\nolimits}
\newcommand{\cones}[3]{\langle v_{#1},\, v_{#2},\, v_{#3}\rangle}
\newcommand{\Rsheaf}{\mathscr{R}}
\newcommand{\Qsheaf}{\mathscr{Q}}
\newcommand{\Ksheaf}{\mathscr{K}}
\newcommand{\Hsheaf}{\mathscr{H}}
\newcommand{\Msheaf}{\mathscr{M}}
\newcommand{\Rei}{\mathcal{R}}
\newcommand{\Rie}{\Rei}
\newcommand{\hol}{\text{hol}}
\newcommand{\Nsheaf}{\mathscr{N}}
\newcommand{\unr}{\text{unr}}
\newcommand{\sHom}{\mathcal{H}om}
\newcommand{\smallmat}[4]{\left(\begin{smallmatrix}{#1} & {#2} \\ {#3} & {#4} \end{smallmatrix} \right)}
\newcommand{\twomat}[4]{\begin{pmatrix}{#1} & {#2} \\ {#3} & {#4} \end{pmatrix}}
\newcommand{\Proh}{\Proj}
\newcommand{\jj}{\mathfrak{j}}
\newcommand{\old}{\text{old}}
\newcommand{\bs}{\backslash}
\newcommand{\diam}[1]{\langle {#1} \rangle}
\newcommand{\Alg}{\textbf{Alg\,}}
\newcommand{\BB}{\mathcal{B}}
\newcommand{\Detla}{\Delta}
\newcommand{\iso}{\xrightarrow{\sim}}
\newcommand{\dep}{\mathop{\mathrm{dep}}\nolimits}
\newcommand{\ind}{\mathop{\mathrm{ind}}\nolimits}
\newcommand{\vol}{\mathop{\mathrm{vol}}\nolimits}
\newcommand{\tr}{\mathop{\mathrm{tr}}\nolimits}
\newcommand{\Stab}{\mathop{\mathrm{Stab}}\nolimits}
\newcommand{\St}{\mathop{\mathrm{St}}\nolimits}
\newcommand{\meas}{\mathop{\mathrm{meas}}\nolimits}
\newcommand{\disc}{\mathop{\mathrm{disc}}\nolimits}
\newcommand{\cusp}{\mathop{\mathrm{cusp}}\nolimits}
\newcommand{\spec}{\mathop{\mathrm{spec}}\nolimits}
\newcommand{\geom}{\mathop{\mathrm{geom}}\nolimits}
\newcommand{\pl}{\mathop{\mathrm{pl}}\nolimits}
\newcommand{\new}{\mathop{\mathrm{new}}\nolimits}

\newcommand{\Lie}{\mathop{\mathrm{Lie}}\nolimits}
\newcommand{\Typ}{\mathop{\mathrm{Typ}}\nolimits}
\newcommand{\diag}{\mathop{\mathrm{diag}}\nolimits}

\newcommand{\trace}{\mathop{\mathrm{trace}}\nolimits}
\newcommand{\ratdeg}{\mathop{\mathrm{ratdeg}}\nolimits}

\newcommand{\inv}{\mathop{\mathrm{inv}}\nolimits}

\newcommand{\mupl}{\widehat \mu^{\mathop{\mathrm{pl}}}}
\newcommand{\mucusp}{\widehat \mu^{\cusp}}
\newcommand{\mudisc}{\widehat \mu^{\disc}}
\newcommand{\GL}{\mathop{\mathrm{GL}}\nolimits}
\newcommand{\PGL}{\mathop{\mathrm{PGL}}\nolimits}
\newcommand{\lev}{\mathop{\mathrm{lev}}\nolimits}

\newcommand{\SL}{\mathop{\mathrm{SL}}\nolimits}
\newcommand{\EP}{\mathop{\mathrm{EP}}\nolimits}
\newcommand{\cind}{\mathop{\mathrm{c\text{-}ind}}\nolimits}
\newcommand{\ad}{\mathop{\mathrm{ad}}\nolimits}

\newcommand{\one}{\mathbf{1}}

\newcommand{\n}{\mathfrak{n}}
\newcommand{\nn}{\mathfrak{n}}
\newcommand{\ff}{\mathfrak{f}}
\newcommand{\oo}{\mathfrak{o}}
\newcommand{\ft}{\mathfrak{t}}
\newcommand{\dd}{\mathfrak{d}}
\newcommand{\fa}{\mathfrak{a}}
\newcommand{\XX}{\mathfrak{X}}

\newcommand{\wh}{\widehat}

\maketitle

\begin{abstract} 
Recent results by Abert, Bergeron, Biringer et al., Finis, Lapid and Mueller, and Shin and Templier have extended the limit multiplicity property to quite general classes of groups and sequences of level subgroups.  Automorphic representations in the limit multiplicity problem are traditionally counted with multiplicity according to the number of fixed vectors of a level subgroup; our goal is to perform a slightly more refined analysis and count only automorphic representations with a given conductor with multiplicity 1.
\end{abstract}

\section{Introduction} 
\label{sec:Intro}
The limit multiplicity problem concerns the asymptotic distribution of the local components of families of automorphic representations.  In particular, it is expected that for families of automorphic representations that arise naturally, the limiting distribution should be the Plancherel measure at the place in question.  

The problem was originally studied by DeGeorge and Wallach \cite{DW78, DW79} who phrased the question in terms of lattices $\Gamma$ in semisimple Lie groups $G$.  They studied the limit multiplicity problem for \emph{normal towers}, or nested sequences of normal subgroups of a fixed maximal lattice whose intersection is the identity; the limit multiplicity problem in this case was completed by Delorme \cite{Del86}.  In recent work, Abert, Bergeron, et al., Finis, Lapid, and Mueller, and Shin and Templier have solved the limit multiplicity problem for a large class of sequences of compact open subgroups in reductive groups $G$ (see, for instance, \cite{ABB+12}, \cite{FL15, FLM14}, \cite{Shi12, ST12}).  In this paper, it is our goal to eliminate a pesky `multiplicity' term from the statement of the Limit Multiplicity problem (at least for forms of $\GL_n$) and to isolate representations of a given conductor, simply counting each with multiplicity 1.

To state our goal, we'll need to clarify the statement of the limit multiplicity problem, following the introduction of \cite{FL15}.  Let $F$ be a number field and let $S_\infty$ denote its set of infinite places.  Let $G/F$ be a reductive algebraic group and let $S \supseteq S_\infty$ be a finite set of places.  We write $F_S = \prod_{v\in S} F_v$ and let $\AAA^S$ be the restricted direct product of $F_\pp$ for $\pp \not \in S$.  We define $G(F_S)^1$ as the intersection of the kernels of the maps $|\chi|_S: G(F_S) \to \RR^\times_{>0}$ where $\chi$ ranges over the $F$-rational characters of $G$; the subgroup $G(\AAA)^1$ is defined similarly.  We define $G(F_S)^{1,\wedge}$ as the space of irreducible, admissible, unitary representations of $G(F_S)^1$.  Fix compatible Haar measures on the groups $G(F_S)^1$, $G(\AAA^S)$, and $G(\AAA)^1$.  Let $K^S \subseteq G(\AAA^S)$ be an open compact subgroup.  Then we define the \emph{counting measure} $\wh\mu_{K^S}$ with respect to $K^S$ on $G(F_S)^{1,\wedge}$ by
$$\wh \mu_{K^S}(\wh f_S) = \frac{\vol(K^S)}{\vol(G(F)\bs G(\AAA)^1)} \sum_{\pi \in G(\AAA)^{1,\wedge}} m_\pi \dim(\pi^S)^{K^S} \wh f_S(\pi_S).$$
Here the sum runs over discrete automorphic $G(\AAA)^1$-representations $\pi$, and $m_\pi$ is the multiplicity of $\pi$ in $L^2(G(F)\bs G(\AAA)^1)$.

The \emph{Plancherel measure} $\mupl$ on $G(F_v)^{1,\wedge}$ for a reductive group $G$ is defined in \cite{Wal03} when $F_v$ is $\pp$-adic, and \cite{Dix77} or \cite{Wal92} when $F_v$ is archimedean; the measure on $G(F_S)^{1,\wedge}$ is the product of these local measures (see \ref{Plancherel} for more details).  Let $\mathscr{F}(G(F_S)^1)$ be the set of bounded, complex-valued functions on $G(F_S)^{1,\wedge}$ that are supported on a finite number of Bernstein components, and that are continuous outside a set of Plancherel measure $0$.  We say that a sequence of subgroups $\{K^S\}$ satisfies the limit multiplicity property if for any $\wh f_S \in \mathscr{F}(G(F_S)^1)$, we have
$$\wh \mu_{K^S}(f_S) \to \mupl(f_S).$$

Recently, the limit multiplicity property has been proven under a wide array of assumptions.  We list some recent work here:
\begin{itemize} 
\item Let $G$ be a Lie group with maximal compact subgroup $K$.  In \cite{ABB+12}, Abert, Bergeron, Biringer et al. proved a limit multiplicity result for sequences $\{\Gamma_n\}$ of lattices in $G$ that are \emph{uniformly discrete} and such that the symmetric spaces $\Gamma_n\bs G / K$ converge to $G / K$ in the sense of \emph{Benjamini-Schramm}. This generalizes the notion of the normal towers of DeGeorge-Wallach and Delorme.
\item In \cite{Shi12}, Shin proved a variant of limit multiplicity for reductive groups $G$ over totally real fields $F$ and for sequences of subgroups that `converge to 1' in an appropriate sense.  An important stipulation was that the test function $f_\infty$ at the real places be an Euler-Poincar\'{e} function; this cuts out a given discrete-series $L$-packet at $\infty$.  The trace formula for such functions is particularly `user friendly'.  A similar argument was given in Section 9 of \cite{ST12}; he and Templier apply this equidistribution result to a result on low-lying zeros of automorphic $L$-functions.
\item In \cite{FLM14}, Finis, Lapid, and Mueller proved the limit multiplicity property for full level subgroups of reductive groups satisfying their conditions (BD) and (TWN) (see Section 5 of \cite{FLM14}), and proved that $\GL_n$ and $\SL_n$ satisfy these conditions.  In \cite{FL14}, Finis and Lapid proved quite general orbital integral bounds, which they used in \cite{FL15} to prove the limit multiplicity property for any sequence of level subgroups whose level goes to $\infty$ (again, under the assumption that the ambient group $G$ satisfies (BD) and (TWN)).
\end{itemize}

It is our goal to prove a slight refinement of the limit multiplicity property in a special case.  Specifically, let $D$ a be group of units in a central division algebra of dimension $n^2$ over a number field $F$ and let $S$ be a finite set of places of $F$, containing the infinite places and also all places at which $D$ does not split.  Then $D(\AAA^S)\cong \GL_n(\AAA^S)$.  If $\pi$ is a discrete automorphic representation and $\pi^S$ is generic then we may discuss the conductor of $\pi^S$ as a $\GL_n(\AAA^S)$-representation.  

In this situation, the conductor has a nice description in terms of open compact subgroups.  Let $\mathbf{K}_{n}^S$ be the maximal compact subgroup $\GL_n(\oo^S)$.  Let $\nn$ be an ideal coprime to $S$ and define the subgroup $K_n(\nn) \leq \mathbf{K}^S$ as the set of matrices
$$\twomat{X}{Y}{Z}{W}$$
where $X \in \mathbf{K}_{n-1}$, $Y$ is an $(n - 1) \times 1$ column vector in of elements of $\oo^S_F$, $Z$ is a $1 \times (n-1)$ row vector of ad\`{e}les divisible by $\nn$, and $W$ is an ad\`{e}le with $W - 1\in \nn$.

It is a classical result of Jacquet, Pietetski-Shapiro, and Shalika (\cite{JPS81}) that a generic representation $\pi$ has a $K_{n}(\nn)$-fixed vector if and only if $c(\pi) \mid \nn$.  Therefore, if we plug $K^S = K_n(\nn)$ into the limit multiplicity problem, the counting measure $\mu_{K^S}$ isolates generic representations $\pi$ whose conductor is divides $\nn$.

We wish to refine this result and isolate the representations whose conductor away from $S$ is \emph{precisely} $\nn$.  In particular, we will construct a test function $e^{\new}_{\nn}$ such that, for a generic representation $\pi$ of $\GL_n(\AAA)$, $\tr \pi(e^{\new}_{\nn}) = 1$ if $\pi$ has conductor $\nn$, and zero otherwise.  An important input into the construction of this function is a result of Reeder (\cite{Ree91}) which counts the dimension of the $K_0(\nn)$-fixed  space in a representation of given conductor.  Our theorem is as follows:

\begin{thm}\label{Thm1} Fix a number field $F$ and let $D/F$ be the group of units of a central division algebra of dimenson $n^2$.  Let $S\supseteq S_\infty$ be a finite set of places, such that $D$ splits at all $\pp\not\in S$.  Assume moreover there is a place $v_0\in S$ at which $D$ splits.  For $\wh f_S\in\mathscr{F}(D(F_S)^1)$ define 
$$\wh\mu_{S,\nn}(\wh f_S) = \frac{1}{e_{\nn}^{\new}(1)\cdot\vol(D(F)\bs D(\AAA)^1)} \sum_{\substack{\pi\\c(\pi^S) = \nn}} \wh f_S(\pi_S)$$
where the sum runs over discrete automorphic representations $\pi$ with $c(\pi^S) = \nn$ such that $\pi^S$ is generic.

Then $\wh \mu_{S,\nn}(\wh f_S) \to \mupl_S(\wh f_S)$ as $N(\nn) \to \infty$.
\end{thm}

The proof follows from an asymptotic bound on the orbital integrals of our test function (see section \ref{sec:4}) and then a relatively standard trace formula argument.

We will also prove a fixed-central-character analog of the limit multiplicity result in this case, following the author's previous work (\cite{Bin15}).  Let $\chi:\AAA^\times \to \CC^\times$ be an automorphic character of conductor $\ff$, and let $\chi_S,\, \chi^S$ be its components at $S$ and away from $S$ respectively; write $\ff^S$ for the conductor of $\chi^S$.  Let $D(F_S)^{\wedge,\chi_S}$ be the subset of $D(F_S)^{\wedge}$ consisting of those representations whose central character is $\chi_S$. Let $\mathscr{F}(D(F_S), \,\chi_S)$ be the space of bounded functions on $D(F_S)^{\chi_S}$ that are supported on finitely many Bernstein components and which are almost-everywhere continuous.  We will construct a test function $e^{\new}_{\nn,\chi}$ in the fixed-central-character Hecke algebra $\HH(D(\AAA^S),\, \chi^S)$ that will cut out generic representations with central character $\chi$.  This will allow us to prove the following fixed-central-character analog of Theorem \ref{Thm1}:

\begin{thm} \label{Thm2} Let $D,\, F,\, S$ be as above and let $\chi: \AAA^\times \to \CC^\times$ be an automorphic character.  Let $\ff^S\mid \nn$ and let $f_S\in \mathscr{F}(D(F_S),\,\chi_S)$.  Set
$$\wh\mu_{S,\nn,\chi}(\wh f_S) = \frac{1}{e_{\nn,\chi}^{\new}(1)\cdot\vol(D(F)Z(\AAA)\bs D(\AAA))} \sum_{\substack{\chi_\pi = \chi  \\ c(\pi^S) = \nn}} \wh f_S(\pi_S)$$
where the sum runs over discrete automorphic representations $\pi$ with central character $\chi$, conductor $\nn$, and such that $\pi^S$ is generic.

Then $\wh \mu_{S,\nn}(\wh f_S) \to \mupl_{S,\chi}(\wh f_S)$ as $N(\nn) \to \infty$, where $\mupl_{S,\chi}$ is the fixed-central-character Plancherel measure on $D(F_S)^{\wedge,\chi}$.
\end{thm}

A brief note on our choice of algebraic group: We have chosen $D$ to be the group of units in a division algebra in order to simplify the trace formula. Since $D$ has no proper parabolic subgroups, then the quotient $D(F)\bs D(\AAA)^1$ is compact and the spectral side of the trace formula consists only of orbital integrals.  In view of the existing work of Finis-Lapid, these will prove easy to bound.  If we choose `more complicated' forms of $\GL_n$ (such as $\GL_n(D)$ or $U(n)$), similar results may be obtained using the same test function and applying a more difficult version of the trace formula.

Our decision to count representations $\pi$ where $\pi^S$ is generic is natural; these are the automorphic $D(\AAA)^1$-representations whose image under the global Jacquet-Langlands functor of \cite{Bad07} is cuspidal.  Moreover, the stipulation that $S$ contain a place $v_0$ at which $D$ splits is a technical condition used to show that the proportion of non-generic representations in the limit multiplicity formula vanishes.

This paper will be organized as follows: in section \ref{sec:2}, we will briefly discuss the basics on Hecke algebras and give the necessary prerequisites.  In section \ref{sec:3}, we will define our test functions $e^{\new}$ for isolating representations of a given conductor.  In section \ref{sec:4}, we will prove bounds on their orbital integrals.  Finally, in section \ref{sec:5} we use the results of the previous sections to prove our main theorems \ref{Thm1} and \ref{Thm2}.

\subsection{Acknowledgements:} I am grateful to my adviser, Sug Woo Shin, for his interest in this project, and to Julee Kim for useful conversations.  Andy Soffer suggested the elegant approach to Proposition 3.2.

\section{Hecke algebras, Plancherel measure, and other prerequisites}
\label{sec:2}
In this section, we briefly discuss Hecke algebras in the fixed- and unfixed-central-character setting, define the Plancherel measure, and state some of the theorems necessary for the argument.

Throughout, $R$ will be used to denote a rather general locally compact ring.  Usually, $R$ will denote a local field, $\AAA$, $F_S$, or $\AAA^S$ where $S$ is a finite set of places of $F$ and $\AAA$ is its ad\`ele ring.

\begin{defns} \label{defn21} Let $D,\,R$ be as above.  The \emph{Hecke algebra} $\HH(D(R)^1)$ is the convolution algebra of complex-valued, compactly supported smooth functions on $D(R)^1$.

If $\pi$ is an irreducible admissible $D(R)^1$-representation and $\phi \in \HH(D(R)^1)$, then the operator
$$\pi(\phi) = \int_{D(R)^1} \phi(g)\pi(g)\,dg$$
is of trace class on $V_\pi$.  We define $\wh \phi(\pi) = \tr \pi(\phi)$.

In the above definitions, $D(R)^1$ may be replaced with $D(R)$.

Fix $\chi: R^\times \to \CC^\times$.  We define the \emph{fixed central character} Hecke algebra $\HH(D(R),\,\chi)$ as the convolution algebra of functions $\phi: D(R) \to \CC$ such that
\begin{itemize} 
	\item $\phi$ is compactly supported modulo the center of $D(R)$, and
	\item For any $g\in D(R),\, z\in Z(R)$ we have
	$$\phi(gz) = \chi^{-1}(z) \phi(g).$$
\end{itemize}

If $\pi$ is an irreducible admissible $D(R)$-representation with central character $\chi$, and $\phi\in \HH(D(R),\,\chi)$, then the integral
$$\pi(\phi) = \int_{D/Z} \phi(g)\pi(g)\,d\bar g$$
is well-defined and of trace class.  We define $\wh \phi(\pi) = \tr \pi(\phi)$.

Finally, given $\chi$, there is an \emph{averaging map}
$$\HH(D(R)^1) \to \HH(D(R),\,\chi),\,\,\,\,\,\,\, \phi_0 \mapsto \overline{\phi_0}$$
such that for $g\in D(R)^1$
$$\overline{\phi_0}(g) = \int_{Z(R)} \phi_0(gz) \chi(z) \,dz;$$
we extend $\overline \phi_0$ to all of $D(R)$ via the transformation property.

The averaging map $\HH(D(R)) \to \HH(D(R),\,\chi)$ is defined similarly.
\end{defns}

Note that the definitions above depend on choices of Haar measure. We will make the following conventions for the rest of the paper: if $\pp$ is a finite place of $F$ and $K$ a maximal compact subgroup of $D(F_\pp)$, then the Haar measure on $D(F_\pp),\, (D/Z)(F_\pp)$ will be chosen to give $K,\, KZ/Z$ measure $1$. At infinite places, the measures may be chosen to be arbitrary; our only stipulation will be that Haar measures on $Z(F_v),\, D(F_v),\, (D/Z)(F_v)$ are chosen to be compatible under the exact sequence
$$1 \to Z(F_v) \to D(F_v) \to (D/Z)(F_v) \to 1.$$
Haar measures on ad\`{e}le groups are chosen as product measures of these local measures.

\begin{lem} \label{lem22} Let $\phi_0\in \HH(D(R))$ and let $\phi\in \HH(D(R),\,\chi)$ image under the averaging map.  Let $\pi$ be an admissible irreducible $D(R)$-representation with central character $\chi$, and pick compatible Haar measures $dz,\, dg,\, d\bar g$ on $Z(R),\, D(R),\, (D/Z)(R)$ respectively.  Then
$$\wh \phi(\pi) = \wh \phi_0(\pi).$$
\end{lem} 
\begin{proof} The proof is a simple application of Fubini's theorem.\end{proof}

In our theorem, the limiting distribution is given by the \emph{Plancherel measure} at a finite set of local places.  We define it here:

\begin{defn} \label{Plancherel} Let $F$ be a number field and let $S$ be a finite set of places.  There is a unique measure $\mupl$ on $D(F_S)^{1,\wedge}$ that is supported on the tempered spectrum $D(F_S)^{1,\wedge,t}$ and such that for any $\phi_S \in \HH(D(F_S))$ we have
$$\int_{D(F_S)^{\wedge}} \wh \phi_S(\pi) \,d\mupl(\pi) = \phi_S(1).$$

Let $\chi: F_S^\times \to \CC^\times$. Then there is a unique measure $\mupl_{\chi}$ on the set $D(F_S)^{\wedge,\chi}$ that is supported on $D(F_S)^{wedge,\chi_S,t}$ and such that for any function $\phi \in \HH(D(F_S),\,\chi)$ we have
$$\int_{D(F_S)^{\wedge,\chi}} \wh \phi_S(\pi)\,d\mupl_{\chi}(\pi) = \phi_S(1).$$
\end{defn}

A self-contained construction of the Plancherel measure in the $\pp$-adic case is given in \cite{Wal03}.  For the real case, references include \cite{Dix77} and \cite{Wal92}.  The existence of a fixed-central-character measure is known to the experts but to our knowledge is not written down fully.  In Proposition 6.2.8 and Subsection 11.2 of \cite{Bin15}, the author constructs such a measure from the non-fixed character case using abelian Fourier analysis in the $\pp$-adic $\GL_2$ case; the same proof goes through in the situation here.

We will need two more group-theoretic theorems.  The first is a density theorem due to Sauvageot.  We'll start with a definition:
\begin{defn} \label{defn24} Let $\mathscr{F}(D(F_S)^1)$ be the space of bounded, complex-valued functions $\wh f$ on $D(F_S)^{\wedge}$ such that $\wh f$ is supported on finitely many Bernstein components, and such that $\wh f$ is continuous outside a set of Plancherel measure zero.  If $\chi: F_S^\times \to \CC^\times$ is a character, define $\mathscr{F}(D(F_S),\chi)$ as the space of such functions $\wh f$ on $D(F_S)^{\wedge,\chi}$.
\end{defn}

As such, the maps $\phi \mapsto\wh \phi$ defined in \ref{defn21} give maps $\HH(D(F_S)^1) \to \mathscr{F}(D(F_S))$ and $\HH(D(F_S),\,\chi) \to \mathscr{F}(D(F_S),\,\chi)$.  The content of the density theorem is that the images are dense in an appropriate sense:

\begin{thm}[Sauvageot's Density Theorem] \label{Sauvageot} Let $\wh f\in \mathscr{F}(D(F_S)^1)$ and fix $\epsilon > 0$.  Then there are functions $\phi,\,\psi \in \HH(D(F_S)^1)$ such that
\begin{itemize} 
	\item $|\wh f(\pi) - \wh \phi(\pi)| \leq \wh \psi(\pi)$ for all $\pi \in D(F_S)^{\wedge}$, and
	\item $\mupl(\wh \psi) < \epsilon$.
\end{itemize}

The analogous theorem holds for $\phi,\,\psi \in \HH(D(F_S), \,\chi)$ if we restrict $\wh f$ to the space $D(F_S)^{\wedge,t,\chi}$.
\end{thm}
\begin{proof} The first statement is Thm. 7.3 of \cite{Sau97}.  The fixed-central-character statement follows by the same logic as in Lemma 11.2.7 of \cite{Bin15}.\end{proof}

\begin{rem} \label{rem26} Sauvageot also shows that if $A$ is a bounded subset of $D(F_S)^{1,\wedge}$ that does not intersect the tempered spectrum, then given $\epsilon$ there is a $\phi\in \HH(D(F_S))$ such that $\wh \phi \geq 0$ everywhere, $\wh \phi \geq 1$ on $A$, and $\mupl(\wh \phi) < \epsilon$.  From this, we discern that if $\wh f \in \FF(D(F_S)^1)$ then so is $\wh f \cdot \one_t$ and $\wh f \cdot (\one - \one_t)$, where $\one_t$ is the characteristic function of the tempered spectrum.  This will be useful later for using Sauvageot's theorem to isolate automorphic representations that are tempered at $S$.\end{rem}

Finally, we'll need the trace formula.  Since $D$ has no proper parabolic subgroups, we may use the Selberg trace formula for compact quotient.

\begin{defn} \label{defn27} Let $\gamma \in D(\AAA)$ and let $\phi\in \HH(D(\AAA)^1)$ or $\HH(D(\AAA),\chi)$.  The \emph{orbital integral} of $\phi$ with respect to $\gamma$ is the quantity
$$O_{\gamma}(\phi) = \int_{D_{\gamma}(\AAA) \bs D(\AAA)} \phi(g^{-1}\gamma g)\, d\bar g.$$
\end{defn}

\begin{thm}[The Selberg Trace Formula] \label{Trace} Let $F$ be a number field with ad\`{e}le ring $\AAA$ and let $\phi \in \HH(D(\AAA)^1)$.  Then we have an equality
\begin{align*} \sum_{\pi} m_\pi \tr \pi(\phi) = & \sum_{z\in Z(F)} \vol(D(F)\bs D(\AAA)^1) \phi(z)
\\ & + \sum_{\gamma \in (D(F) - Z(F))/\sim} \vol(D_{\gamma}(F) \bs D_{\gamma}(\AAA)^1) O_{\gamma}(\phi).
\end{align*}
where the left-hand sum runs over all automorphic representations of $D(\AAA)$, and the right-hand side runs over conjugacy classes in $D(F)$.

If $\chi: \AAA^\times \to \CC^\times$ is an automorphic character and $\phi\in\HH(D(\AAA),\,\chi)$, then 
\begin{align*} \sum_{\chi_\pi = \chi} m_\pi \tr \pi(\phi) = & \vol(D(F)Z(\AAA)\bs D(\AAA)) \phi(1) 
\\ & + \sum_{\gamma \in (D(F) - Z(F))/\sim} \vol(D_{\gamma}(F)Z(\AAA) \bs D_{\gamma}(\AAA)) O_{\gamma}(\phi).
\end{align*}
Here the left-hand side runs over automorphic representations whose central character is $\chi$, and the right-hand side runs over equivalence classes of elements in $D(F)$, where $\gamma$ and $\gamma'$ are equivalent if $\gamma'$ is conjugate to $z\gamma$ for $z\in Z(F)$.
\end{thm}

\section{The `new vector' test function}
\label{sec:3}

Let $F,\, D,\, S$ be as in the previous section and let $\nn$ be an ideal of $\oo_F$, coprime to $S$.  Throughout, $\chi$ will denote a character $\AAA^{S,\times} \to \CC^\times$ whose conductor $\ff$ divides $\nn$.  In this section, we will construct explicit test functions $e^{\new}_{n,\nn} \in \HH(\GL_n(\AAA^S))$ and $e^{\new}_{n,\nn,\chi}\in \HH(\GL_n(\AAA^S),\,\chi)$ such that, for any \emph{generic} representation $\GL_n(\AAA^S)$-representation $\pi^S = \bigotimes_{\pp\not\in S} \pi_\pp$, we have
$$\wh e^{\new}_{n,\nn}(\pi^S) = \tr \pi^S(e^{\new}_{n,\nn}) = 
	\begin{cases}
	1 & c(\pi^S) = \nn \\
	0 & c(\pi^S) \neq \nn.
	\end{cases}$$
and similarly for $e^{\new}_{n,\nn,\chi}$ when $\pi^S$ has central character $\chi$.
	
We'll construct $e^{\new}$ as a product of local test function $e^{\new}_{n,\pp^r}$ for $\pp^r \mid \nn$; we'll later construct $e^{\new}_{n,\nn,\chi}$ as the image of $e^{\new}_{n,\nn}$ under the averaging map $\HH(\GL_n(\AAA^S)) \to \HH(\GL_n(\AAA^S),\chi)$.  We'll need two inputs: a theorem of Reeder and a combinatorial identity.  Recall the definition of $K_n(\pp^r)$ from the introduction.

\begin{thm}[\cite{Ree91}, Theorem 1] \label{Reeder} Let $\pi_\pp$ be a generic irreducible admissible representation of $\GL_n(F_{\pp})$ of conductor $c = c(\pi_p)$.  Then
$$\dim \pi_p^{K_n(\pp^r)} = \binom{r - c + n - 1}{n - 1}$$
\end{thm}

It's worth remarking that the genercity condition is necessary. For example, if $\pi$ is the trivial representation then $\dim \pi^{K_n(\pp^r)} = 1$ for all $r$.

\begin{prop} \label{prop32} For any $n\in \ZZ_{\geq 1}$ and $k \in \ZZ$, the following identity holds:
$$\sum_{i = 0}^n (-1)^{i} \binom{n}{i} \binom{k - i + n - 1}{n - 1} = 
	\begin{cases} 1 & \text{if }k = 0
	\\ 0 & \text{otherwise}. \end{cases}$$
\end{prop}
\begin{proof} If $k = 0$ then the only nonzero term of the right-hand side is the $i = 0$ term, which is $1$.  If $k < 0$ then all terms of the sum are zero.

If $k > 0$, consider the polynomial function $g_{k,n}(x) = (x-1)^n x^{k-1}$.  This polynomial vanishes with order $n$ at $x = 1$, so $g^{(n-1)}_{k,n}(1) = 0$.  On the other hand, we may expand $g_{k,n}$ as 
$$g_{k,n}(x) = \sum_{i = 0}^n (-1)^i\binom{n}{i} x^{k - 1 + n - i}$$
so that
$$g^{(n-1)}_{k,n}(x) = (n-1)! \sum_{i = 0}^{n}  (-1)^i \binom{n}{i} \binom{k-1+n-i}{n-1} x^{k - 1 + n - i}$$
(note that if $k - 1 + n - i < 0$ then $\binom{k-1+n-i}{n-1} = 0$) and therefore 
$$0 = g^{(n-1)}_{k,n}(1) = (n-1)! \sum_{i = 0}^n (-1)^i \binom{n}{i} \binom{k-1+n-i}{n-1}$$
completing the proof.
\end{proof}

This motivates the following definition:
\begin{defn} \label{defn33} Given a prime $\pp$ and a conductor $r$, let 
$$e^{\new}_{n,\pp^r} = \sum_{i = 0}^n (-1)^i \binom{n}{i} e_{K_n(\pp^{r-i})} \in \HH(\GL_n(F_{\pp}))$$
where $e_{K_n(\pp^{r-i})}\in \HH(\GL_n(F_\pp))$ is the idempotent function corresponding to the open compact subgroup $K_n(\pp^{r-i})$ of $\GL_n(F_{\pp})$.  (By abusing notation, if $r- i < 0$, we set $e_{K_n(\pp^{r-i})} = 0$).

If $\nn = \prod_{\pp} \pp^r_{\pp}$, define
$$e^{\new}_{n,\nn} = \left(\prod_{\pp\mid \nn} e^{\new}_{n, \pp^{r_\pp}}\right) \times \left(\prod_{\pp\nmid \nn} \one_{\mathbf{K}_p}\right).$$
\end{defn}

\begin{prop} \label{prop34} \begin{enumerate}[(1)]
	\item Let $\pi_\pp$ be a generic representation of $\GL_n(F_\pp)$.  Then
		$$\tr \pi_{\pp}(e^{\new}_{\pp^r}) = 
			\begin{cases} 1 & c(\pi_{\pp}) = r
			\\ 0 & \text{otherwise}. \end{cases}$$
	\item Let $\pi^S$ be a generic automorphic representation of $\GL_n(\AAA^S)$ and let $\nn$ be an ideal coprime to $S$.  Then
	$$\tr \pi^S(e^{\new}_{\nn}) = 
		\begin{cases} 1 & c(\pi^S) = \nn
		\\ 0 & \text{otherwise}.\end{cases}$$
\end{enumerate}
\end{prop}
\begin{proof} The first statement follows from Reeder's theorem, our combinatorial identity \ref{prop32}, and the fact that if $K \leq \GL_n(F_{\pp})$ is an open compact subgroup, then $\tr \pi_{\pp} (e_{K}) = \dim \pi_{\pp}^K$.  The second statement follows directly from the first.
\end{proof}

If $\chi^S: (\AAA^S)^\times \to \CC^\times$ is a character of conductor $\ff^S$ and $\ff^S \mid \nn$ we define
$e^{\new}_{n,\nn,\,\chi}$ to be the image of $e^{\new}_{n,\nn}$ under the averaging map $\HH(\GL_n(\AAA^S)) \to \HH(\GL_n(\AAA^S),\,\chi)$.  The following corollary follows immediately from Proposition \ref{prop34} and Lemma \ref{lem22}.

\begin{cor} \label{cor34} Let $\pi^S$ be a generic automorphic representation of $\GL_n(\AAA^S)$ with central character $\chi^S$ and let $\nn$ be an ideal coprime to $S$ and divisible by $\ff$. 
Then
	$$\tr \pi^S(e^{\new}_{\nn,\chi}) = 
		\begin{cases} 1 & c(\pi^S) = \nn
		\\ 0 & \text{otherwise}.\end{cases}$$
\end{cor}

\section{Asymptotic vanishing of orbital integrals} 
\label{sec:4}

Let $h_{n,\nn}(g) = e^{\new}_{n,\nn}(g)/e^{\new}_{n,\nn}(1)$; define $h_{n,\nn,\chi}$ similarly.  In the next section, we will prove Theorems \ref{Thm1} and \ref{Thm2} by plugging a test function of the form
$$h \otimes \phi_S$$
into the trace formula.  To make the argument run, we will need to prove the asymptotic vanishing of the orbital integrals
$$O_{\gamma^S}(h_{n,\nn})$$
(see \ref{defn27}) for noncentral elements $\gamma\in D(F)$.  We'll prove the analogous result for $h_{n,\nn,\chi}$ in the following subsection.

In particular, the goal of this section is to prove the following:
\begin{prop} \label{prop41} Let $F,\,D,\,S$ be as above.  There are constants $C(\gamma),\,\epsilon > 0$ such that, for every non-central element $\gamma \in D(F)$ and any ideal $\nn$ coprime to $S$, we have: 
$$|O_{\gamma^S}(h_{n,\nn})| \leq C(\gamma)(3(n+1))^{P(\nn)} N(\nn)^{-\epsilon}.$$

In particular, for fixed $\gamma$, we have $O_{\gamma^S}(h_{n,\nn}) \to 0$ as $N(\nn) \to \infty$.
\end{prop}

We'll begin with a lemma that bounds $|h_{n,\pp^r}|$ by a linear combination of characteristic functions of compact open subgroups $K_{n}(\pp^{r-i}) \leq \mathbf{K}$.

\begin{lem} \label{lem42} Fix a prime $\pp$ of norm $q$.  For every $n \geq 2$ and conductor $r$ we have
$$|h_{n,\pp^r}| \leq 3\cdot \sum_{i = 0}^{n} q^{-ni} \one_{K_n(\pp^{r-i})}$$
(as in Definition \ref{defn33} we replace $\one_{K_n(\pp^{r-i})}$ with the zero function if $r - i < 0$).
\end{lem}
\begin{proof} We note that $e_{K(\pp^r)}(1)$ is the inverse of the Haar measure of $K_n(\pp^r)$, or $q^{n(r-1)}(q^{n} - 1)$.  Therefore, it suffices to show that $e^{\new}_{n, r, F_{\pp}}(1) \geq \frac{1}{3} \cdot q^{n(r-1)}\cdot (q^n - 1)$.

We therefore write:
\begin{align*} e^{\new}(1)
& = (q^n - 1)\left(q^{n(r-1)} - \binom{n}{1}q^{n(r-2)} + \binom{n}{2} q^{n(r-3)} - \ldots \right) [\pm 1]
\\ & \geq (q^n - 1)q^{(r-1)n} \left(1 - \frac{4}{3}\left(\binom{n}{1}q^{-n} + \binom{n}{3} q^{-3n} + \ldots\right)\right).\end{align*}
Here the constant 4/3 is necessary to deal with the fact that $\frac{\one_{K_n(\pp)}(1)}{\one_{K}(1)} = q^{n} - 1$, rather than $q^n$ (and $\frac{q^{n} - 1}{q^n} \geq 3/4$ for $q,\,n \geq 2$).  The term $[\pm 1]$ may or not occur.

We have moreover that
$$\sum_{i = 1}^{\lceil n/2\rceil} \binom{n}{2i - 1} q^{-(2i - 1)n} = \frac{1}{2}\left((1 + q^{-n})^{n} - (1 - q^{-n})^n \right);$$
for $q,\,n\geq 2$, this quantity is bounded above by $1/2$, completing the proof.
\end{proof}

Taking the product over local places gives the following global bound:
\begin{lem} \label{lem43} Let $\nn$ be an ideal of $\oo_F$ and let $P(\nn)$ be the number of primes dividing $\nn$.

$|h_{n,\nn}|$ is bounded above by a function of the form
$$3^{P(\nn)} \sum_{\dd}\left(\frac{N(\dd)}{N(\nn)}\right)^{n}\one_{K_n(\dd)}$$
for a set of ideals $\dd$ dividing $\nn$.  Moreover, the number of terms in the sum is bounded above by $(n+1)^{P(\nn)}$.
\end{lem}

Therefore, to prove Proposition \ref{prop41}, it suffices to prove:
\begin{lem} \label{lem44} Let $\gamma\in D(F)$ be noncentral.  Then
$$|O_{\gamma}(\one_{K_n(\dd)})| \leq C(\gamma) N(\dd)^{-\epsilon}.$$
\end{lem}

We'll use the orbital integral bounds of Finis-Lapid to prove the lemma.  First, we recall some of their notation and adapt it to our situation:

\newcommand{\fg}{\mathfrak{g}}
\begin{defn} \label{defn45} Let $\fg$ be the Lie algebra of $D$ and fix an isomorphism $\GL(\fg) \cong \GL_{n^2}$.  Fix $x_p\in \mathbf{K}_p \leq D(F_\pp)$.  The quantity $\lambda_{\pp}(x_p)$ is the largest $n\in \ZZ_{\geq 0} \cup \{\infty\}$ such that $\ad(x_p) \in \GL_{n^2}(F_{\pp})$ lies in the full level subgroup $\Gamma(\pp^n)$.

For $x^S\in \GL_n(\AAA^S)$, we define $\lambda(x) = \prod_{\pp\not\in S} \pp^{\lambda_{\pp}(x_\pp)}$.
\end{defn}

We leave it to the reader to check the equivalence between this definition and definition (5.2) in \cite{FL14} in the case where $G = \GL_n$. 
\begin{lem} \label{lem46} Let $\gamma \in D(F)$ be noncentral, and let $S$ contain all the infinite places and all the places at which $D$ does not split.  There is an ideal $\nn(\gamma)$ such that for any $g\in D(\AAA^S)$ we have $\lambda(g\gamma g^{-1}) \mid \nn(\gamma)$.
\end{lem}

\begin{proof} Let $A/F$ be a central simple algebra such that $D \cong A^\times$ and let $f_\gamma(t)\in F[t]$ be the characteristic polynomial of $\ad(\gamma)$ acting on $A$.  If $g\gamma g^{-1} \equiv \lambda \mod \nn$, for a central element $\gamma \in D(\AAA^S)$, then we would have $f_\gamma(t) \cong (t - 1)^{n^2} \mod \nn$.

Since $\gamma$ is noncentral and the action of $\ad(\gamma)$ on $A$ is semisimple, then $f_{\gamma}(t) \neq (t-1)^{n^2}$; in particular, there is a smallest ideal $\nn(\gamma)$ such that $f_\gamma(t) \cong (t-1)^{n^2} \mod \nn(\gamma)$.  Therefore, $\lambda(g\gamma g^{-1})$ must divide the ideal $\nn(\gamma)$ for all $g$, completing the proof.
\end{proof}

\begin{lem} \label{lem47} Fix $\gamma \in D(F)$; by abuse of notation we will identify $\gamma$ with its image in $\GL_n(\AAA^S)$.  Then there is an $\epsilon > 0$ such that, for any level subgroup $K'$ of $\mathbf{K}^S = \GL_n(\oo_{F}^S)$ we have
$$O_{\gamma^S}(\one_{K'}) \ll lev(K')^{-\epsilon}.$$
\end{lem}

\begin{proof} Pick $\gamma_1,\ldots,\, \gamma_r \in \mathbf{K}^S$ that are conjugate to $\gamma$ in $\GL_n(\AAA^S)$, but such that $\gamma_i,\, \gamma_j$ are not conjugate by an element of $\mathbf{K}^S$.  We may pick a finite set because the orbital integral $O_{\gamma^S}(\one_{\mathbf{K}^S})$ is finite.

By the previous lemma, each $\gamma_i$ satisfies $\lambda(\gamma_i) < \infty$.  As such, the measure of the set
$$\{k\in \mathbf{K}^S: k\gamma_ik^{-1}\in K'\}$$
is bounded by $C \lev(K')^{-\epsilon}$, by Remark 5.4 of \cite{FL15} (we can assume $C$ does not depend on $i$ since there are finitely many $\gamma_i$).  This gives an upper bound
$$O_{\gamma}(\one_{K'}) \leq rC\lev(K')^{-\epsilon}.$$\end{proof}

Now Lemma \ref{lem44} follows as a corollary, once we note that $\lev(K_n(\dd)) = \dd$.

\subsection{The fixed-central-character case}
The analysis of the fixed-central-character test function is slightly more difficult, so we have opted to complete that case in this subsection.

We'll need a description of $e^{\new}_{\nn,\chi}$ as a product of local functions.  Recall that $e^{\new}_{\pp^r}$ is given by a linear combination of idempotent functions $e_{K(\pp^r)}$; let ${e}_{\pp^r,\chi}$ be their images in $\HH(\GL_n(F_\pp),\,\chi)$ under the averaging map.  Let $K'(\pp^r)$ be the set of matrices
$$\twomat{X}{Y}{Z}{W} \in \GL_n(F_{\pp})$$
with $X\in \GL_{n-1}(\oo_\pp)$, $Y$ is an $(n-1)\times 1$-vector of elements of $\oo$, $Z$ is a $1 \times (n-1)$ vector of elements in $\pp^r$, and $W\in \oo_\pp^\times$.  

The following lemma is an easy computation:
\begin{lem} \label{lem48}
	\begin{enumerate}[(i)] 
		\item ${e}_{\pp^r,\chi}$ is supported on $\mathbf{K}_{\pp}\cdot Z$
		\item $\mathbf{K}_{\pp} \cap \supp({e}_{\pp^r,\chi}) = K'(\pp^r)$
		\item For our choice of Haar measure,
			$$|{e}_{\pp^r,\chi}(g)| = [\mathbf{K}_\pp: K'(\pp^r)] = \frac{q^{n} - 1}{q - 1} q^{(r-1)(n-1)}$$	
			for any $g\in K'(\pp^r)$.
\end{enumerate}
\end{lem}
		
As above let $h_{n,\pp^r,\chi} = e^{\new}_{n,\pp^r,\chi}/e^{\new}_{n,\pp^r,\chi}(1)$.  We have the following analog of Lemma \ref{lem42}:
\begin{lem} \label{lem49} Given a prime $\pp$ of norm $q$, $n\geq 2$ and conductor $r$, we have
$$|h_{n,\pp^r,\chi}| \leq 6 \sum_{i = 0}^n q^{-(n-1)i} \one_{Z\cdot K'_n(\pp^{r-i})},$$
where we take the characteristic function to be zero if $r - i < 0$.
\end{lem}
\begin{proof} 
As in the proof of Lemma \ref{lem42}, we need to show that $e^{\new}_{\pp^r,\chi}(1) \geq \frac{1}{6} \frac{q^n - 1}{q-1} q^{(n-1)(r-1)}$.  We compute
\begin{align*} 
	e^{\new}(1) 
		& = \frac{q^n -1}{q-1} \left(q^{(r-1)(n-1)} - \binom{n}{1} q^{(r-2)(n-1)} + \binom{n}{2}q^{(r-3)(n-1)} - \ldots \right)
		\\ & \geq \frac{q^n - 1}{q - 1} q^{(r-1)(n-1)} \left(1 - \binom{n}{1} q^{-(n-1)} - \binom{n}{3} q^{-3(n-1)} - \ldots\right)
\end{align*}
As above, consider the function
$$g(n,\,q) = \sum_{i=1}^{\lceil n/2\rceil} \binom{n}{2i - 1} q^{-(2i - 1)(n-1)} = \frac{1}{2} \left((1 + q^{1-n})^n - (1 - q^{1-n})^n\right).$$

Here we need to be a bit careful, since $g$ is not uniformly bounded away from $1$ when $n,\,q \geq 2$.  However, by taking derivatives we can see that this quantity is decreasing in $q$ and $n$ in the region where $q,\,n\geq 2$.  When $q  = 2$ and $n = 3$, this quantity is $\frac{49}{64} < 5/6$, and when $q = 3$ and $n = 2$ the quantity is $2/3$.  We'll examine the case $n = q = 2$ separately.

In the case $r \geq 3$ then 
$$e^{\new}(1) = (q+1)q^{r-1}\left(1 - 2q^{-1} + q^{-2}\right) = \frac{1}{4} (q+1)q^{r-1}.$$
If $r = 2$ then 
$$e^{\new}(1) = (q+1) q \left(1 - 2q^{-1} + \frac{1}{q(q+1)}\right) = \frac{1}{6}(q+1)q.$$
If $r = 1$ then
$$e^{\new}(1) = (q+1) \left(1 - \frac{2}{q+1}\right) = \frac{1}{3} (q+1).$$
Finally, if $r = 0$ then $e^{\new}(1) = 1$.

As such, for any $\pp,\,n,\,r,\,\chi$ we have $e^{\new}_{\pp^r,\chi}(1) > \frac{1}{6} e_{Z\cdot K'(\pp^r)}$.
\end{proof}

With this in hand, we can prove the asymptotic vanishing of orbital integrals:

\begin{prop} \label{prop410} Let $\gamma \in D(F)$ be noncentral.  Then there are constants $C(\gamma),\,\epsilon > 0$ such that, for any ideal $\nn$ coprime to $S$, we have
$$|O_{\gamma^S}(h_{n,\nn,\chi})| \leq C(\gamma)(6(n+1))^{P(\nn)} N(\nn)^{-\epsilon}.$$
In particular, $O_{\gamma^S}(h_{n,\nn,\chi}) \to 0$ as $N(\nn) \to \infty$.
\end{prop}
\begin{proof} Given lemma \ref{lem49}, the proof is essentially the same as the proof of Proposition \ref{prop41}.  The only extra piece we need is this: fix $\gamma \in D(F)$ such that $O_{\gamma^S}(\one_{\mathbf{K}^S})$ is nontrivial.  By conjugating and shifting by an element of the center, we can in fact assume $\gamma \in \mathbf{K}^S$.  Now if $g\gamma g^{-1} \in Z\cdot \mathbf{K}^S$, we must have $g \gamma g^{-1} \in \mathbf{K}^S$ by taking determinants.  As such, if $K\leq \mathbf{K}^S$ with $ZK \cap \mathbf{K}^S = K'$ then $O_{\gamma^S}(\one_{ZK}) =  O_{\gamma^S}(\one_{K'})$.

By the result of Lemma \ref{lem49}, we may bound $|h_{\nn,\chi}|$ by a sum of functions of the form $\frac{N(\dd)}{N(\nn)} 6^{P(\nn)} \one_{ZK'(\dd)}$.  This result, together with an orbital integral bound analogous to Lemma \ref{lem44}, completing the proof.
\end{proof}

\section{Proof of the refined limit multiplicity proof}
\label{sec:5}
In this section, we will prove our primary Theorems \ref{Thm1} and \ref{Thm2}:

\begin{thm}
\begin{enumerate}[(1)] Let $F$ be a number field and $D/F$ the group of units in a division algebra.  Let $S\supseteq S_\infty$ be a finite set of places such that $D$ splits at all $\pp\not\in S$ and also splits at at least one $v_0\in S$.
 \label{Thm3}
\item For $\wh f_S \in \mathscr(D(F_S)^1)$, let
$$\wh\mu_{S,\nn}(\wh f_S) = \frac{1}{e^{\new}_{\nn}(1) \vol(D(F)\bs D(\AAA)^1)} \sum_{\substack{\pi\\c(\pi^S) = \nn}} \wh f_S(\pi_S)$$
where the sum runs over automorphic $D(F_S)^1$ representations $\pi$ such that $\pi^S$ is generic and $c(\pi^S) = \nn$.

Then $\wh \mu_{S,\nn}(\wh f_S) \to \mupl_S(\wh f_S)$ as $N(\nn) \to \infty$.

\item Fix an automorphic character $\chi:\AAA^\times \to \CC^\times$ with conductor $\ff$.  For $\nn$ coprime to $S$ and divisible by $\ff^S$, and $\wh f_{S} \in \mathscr{F}(D(F_S),\chi)$, let
$$\wh\mu_{S,\nn,\chi}(\wh f_S) = \frac{1}{e^{\new}_{\nn,\chi}(1) \vol(Z(\AAA)D(F)\bs D(\AAA))} \sum_{\substack{\chi_\pi = \chi \\ c(\pi^S) = \nn}} \wh f_S(\pi_S).$$

Then $\wh \mu_{S,\nn,\chi}(\wh f_S) \to \mupl_{S,\chi_S}(\wh f_S)$ as $N(\nn)\to \infty$.
\end{enumerate}
\end{thm}

We'll begin with a slightly weaker result:
\begin{prop}  \label{prop52}
\begin{enumerate}[(1)] 
\item Let $\wh f_S \in\mathscr{F}(D(F_S)^1)$.  Then
$$\lim_{N (\nn) \to \infty} \frac{1}{e^{\new}_{n,\nn}(1) \vol(D(F)\bs D(\AAA)^1)} \sum_{\pi} \wh{e}^{\new}_{n,\nn}(\pi^S) \wh f_S(\pi_S) = \mupl_S(\wh f_S).$$
\item Let $\chi$, $\nn$ be as above and fix $\wh f_S \in \mathscr{F}(D(F_S),\chi)$.  Then
$$\lim_{N (\nn) \to \infty} \frac{1}{e^{\new}_{n,\nn,\chi}(1) \vol(Z(\AAA)D(F)\bs D(\AAA))} \sum_{\chi_\pi = \chi} \wh{e}^{\new}_{n,\nn,\chi}(\pi^S) \wh f_S(\pi_S) = \mupl_{S,\chi}(\wh f_S).$$
\end{enumerate}
\end{prop}
\begin{proof} 
We'll prove $(1)$ first; the proof of $(2)$ will be analogous. Let $h_{n,\nn}$ be as in Section \ref{sec:4} and consider a test function of the form $h_{n,\nn} \otimes \phi_S$, where $\phi_S \in \HH(D(F_S)^1)$.  We will first prove the theorem in the case where $\wh f_S = \wh \phi_S$.   

Using the trace formula, we have
\begin{align*} \sum_{\pi} m_{\pi} \wh h_{n,\nn}(\pi^S) \wh \phi_S(\pi_S) = &
	\sum_{z\in Z(F)} \vol(D(F) \bs D(\AAA)^1) h_{n,\nn}(z) \phi_S(z)
	\\ & + \sum_{\gamma \in (D(F) - Z(F))/\sim} O_{\gamma^S}(h_{n,\nn}) O_{\gamma_S}(\phi_S)
\end{align*}

For the second sum, there are only finitely many nonvanishing orbital integrals since the functions $h_{n,\nn}\phi_S$ are uniformly supported on $\mathbf{K}^S \times \supp(\phi_S)$, and this support intersects only finitely many conjugacy classes of elements of $D(F)$.  Each of them vanishes asymptotically by Proposition \ref{prop41}, so the second sum goes to zero.

For the first sum, assume $z\neq 1$; we'll show $h_{n,\nn} \to 0$ as $n\to\infty$.  We note that $\one_{K_n(\dd)}(z) = 1$ if and only if $z - 1 \in \dd$.  Since
$$|h_{\nn,n}| \leq 3 \sum_{\dd \mid \nn} \frac{N(\dd)}{N(\nn)} \one_{K_n(\dd)}$$
then we have
$$|h_{\nn,n}(z)| \leq d(z - 1) \frac{N_{F^S}(z-1)}{N(\nn)}$$
which approaches zero as $N(\nn) \to \infty$ (here  $d(z-1)$ is the number of ideals dividing the ideal $(z-1)$.)

As such, we have
\begin{align*} \lim_{N(\nn) \to \infty} \sum_{\pi} m_\pi h_{n,\nn}(\pi^S) \wh \phi_S(\pi_S) 
	& = \vol(D(F) \bs D(\AAA)^1) \cdot h_{n,\nn}(1)\cdot \wh \phi_S(1) 
\\ & = \vol(D(F) \bs D(\AAA)^1)\cdot \mupl(\wh\phi_S)\end{align*}
so that
$$\lim_{N (\nn) \to \infty} \frac{1}{e^{\new}_{n,\nn}(1) \vol(D(F)\bs D(\AAA)^1)} \sum_{\pi} m_\pi \wh e^{\new}_{n,\nn}(\pi^S) \wh\phi_S(\pi_S) = \mupl(\wh \phi_S).$$
Moreover, $m_\pi = 1$ by the multiplicity one theorem of \cite{Bad07}.

With this in hand, we may prove the same result when $\wh f_S$ is an arbitrary element of $\mathscr{F}(D(F_S)^1)$ (this argument is by now standard and is repeated here for completeness).  For simplicity, we may write the quantity inside the limit as $I_{\spec}(\nn,\, \wh\phi_S)$.  Fix $\epsilon > 0$.  Using Sauvageot's density theorem we may find functions $\phi_S,\, \psi_S \in \HH(D(F_S)^1)$ such that
\begin{itemize}
	\item For all $\pi_S \in D(F_S)^{1,\wedge}$, we have $|\wh f_S(\pi_S) - \wh \phi_S(\pi_S)| \leq \wh\psi_S(\pi_S)$, and
	\item $\mupl(\psi_S) < \epsilon/4$.
\end{itemize}

If $\nn$ is sufficiently high, we have $|I_{\spec}(\nn,\,\wh\phi_S) - \mupl(\wh \phi_S)|,\, |I_{\spec}(\nn,\,\wh\psi_S) - \mupl(\wh \psi_S)| < \epsilon/4$.  Then we have
\begin{align*} |I_{\spec}(\nn,\,\wh f_S) - \mupl(\wh f_S)| & 
	\leq |I_{\spec}(\nn,\, \wh f_S) - I_{\spec}(\nn,\, \wh \phi_S)| + 
	|I_{\spec}(\nn,\, \wh \phi_S) - \mupl(\wh\phi_S)| +
	|\mupl(\wh\phi_S) - \mupl(\wh f_S)|
	\\ & \leq |I_{\spec}(\nn,\,\wh\psi_S)| + |I_{\spec}(\nn,\, \wh \psi_S) - \mupl(\wh\phi_S)| + |\mupl(\wh\psi_S)|
\end{align*}
The first term is at most $\epsilon/2$, since $\mupl(\wh \psi_S) \leq \epsilon/4$ and $ |I_{\spec}(\nn,\,\wh\psi_S) - \mupl(\wh\psi_S)| < \epsilon/4$.  The second and third terms are both bounded by $\epsilon/4$, so we have 
$$|I_{\spec}(\nn,\,\wh f_S) - \mupl(\wh f_S)| < \epsilon$$
completing the proof.

The proof of (2) is entirely analogous, except that we refer to Proposition \ref{prop410} instead of Proposition \ref{prop41}, and there is only one central term in the trace formula.  
\end{proof}

With this in hand, we are ready to complete the proof of the theorem.  Following Remark \ref{rem26}, the above argument proves the following:

\begin{cor} Fix $\wh f_S$ and write $\wh f_S = \wh f_S \cdot \one_t + \wh f_S \cdot \one^c_{t}$, where $\one_t$ is the characteristic function of the tempered spectrum and $\one_{t}^c$ is the characteristic function of its complement.  Then
$$\lim_{N (\nn) \to \infty} \frac{1}{e^{\new}_{n,\nn}(1) \vol(D(F)\bs D(\AAA)^1)} \sum_{\pi} \wh{e}^{\new}_{n,\nn}(\pi^S)\cdot (\wh f_S\cdot \one_t)(\pi_S) = \mupl(\one_t \cdot \wh f_S) = \mupl(\wh f_S)$$
and
$$\lim_{N (\nn) \to \infty} \frac{1}{e^{\new}_{n,\nn}(1) \vol(D(F)\bs D(\AAA)^1)} \sum_{\pi} \wh{e}^{\new}_{n,\nn}(\pi^S)\cdot(\wh f_S\cdot \one_t^c) (\pi_S) = \mupl(\one_t^c \cdot \wh f_S) = 0.$$

The analogous result holds in the fixed-central-character case.
\end{cor}

We now proceed to the proof of the theorem.

\begin{proof} Fix $\pi$ such that $\wh h_{n,\nn}(\pi^S) \wh f_S(\pi_S) \neq 0$, and consider the image $\pi'$ under the Jacquet-Langlands functor of \cite{Bad07}.  If $\pi'$ is cuspidal, then $\pi'^S \cong \pi^S$ is generic everywhere.  If $\pi'$ is not cuspidal, then it follows from \cite[Theorem 4.3]{Wal84} (in the archimedean case) and \cite[Proposition 4.10]{Clo90} (in the non-archimedean case) that $\pi$ is not tempered at any split place and in particular it is not tempered at the place $v_0\in S$.  (In the case of $\GL_n$-representations, this can also be seen directly from the characterization of the residual spectrum in \cite{MW89}).

As such, we have
\begin{align*} \mupl(\wh f_S) & 
	= \lim_{N (\nn) \to \infty} \frac{1}{e^{\new}_{n,\nn}(1) \vol(D(F)\bs D(\AAA)^1)} 
		\sum_{\pi} \wh e^{\new}_{n,\nn}(\pi^S) (\wh f_S\cdot \one_t)(\pi_S)
	\\ & = \lim_{N (\nn) \to \infty} \frac{1}{e^{\new}_{n,\nn}(1) \vol(D(F)\bs D(\AAA)^1)} 
		\sum_{\pi^S \text{ generic}} \wh e^{\new}_{n,\nn}(\pi^S) \wh f_S(\pi_S)
	\\ & = \lim_{N (\nn) \to \infty} \frac{1}{e^{\new}_{n,\nn}(1) \vol(D(F)\bs D(\AAA)^1)} 
		\sum_{\substack{\pi^S \text{ generic} \\ c(\pi^S) =\nn} } \wh f_S(\pi_S)
\end{align*}
in the non-fixed-central-character case.  The fixed-central-character case is entirely analogous.
\end{proof}

\end{document}